\documentclass[11pt]{amsart} % [...,a4paper]
\usepackage{amssymb, amsthm, amsfonts, url} % ,a4wide
\newtheorem{lemma}{\bf Lemma}[section]

\newtheorem{corollary}[lemma]{\bf Corollary}
\newtheorem{proposition}[lemma]{\bf Proposition}

\newtheorem{definition}[lemma]{\bf Definition}

\begin{document}

%...... title
\title{Problems 85 and 87 of Birkhoff's {\it lattice theory}}

%...... authors
\author{Jonathan David Farley}
\address{Morgan State University, Baltimore MD, United
States of America}
\email{\tt jonathan.farley@morgan.edu}
\author{Dominic van der Zypen}
\address{Federal Office of Social Insurance, CH-3003 Bern,
Switzerland}
\email{\tt dominic.zypen@gmail.com}

%......MSC subject class

%...... Abstract...
\begin{abstract} 
We solve problems 85 and 87 from Birkhoff's book {\it Lattice Theory} \cite{Bir}.
\end{abstract}

%....... main()
\maketitle
\parindent = 0mm
\parskip = 2 mm
% . . . . . . . . . . . . . . . . .
\section{Introduction}

A {\em partially ordered set} (or {\em poset} for short) 
is a set $X$ with a binary
relation $\leq$ that is reflexive, transitive, and
anti-symmetric (i.e., $x,y\in X$ with $x\leq y$ and $y\leq x$
implies $x=y$). Often, a poset is denoted by $(X,\leq)$. 
A subset $D\subseteq X$ is called a {\em down-set} if it is
``closed under going down'', that is $d\in D, x\in X, x\leq d$
jointly imply $x\in D$. A special case of a down-set
is the set $$\downarrow_P x = \{y\in X: y\leq x\}$$
for $x\in X$. (Sometimes we just write $\downarrow x$
if the poset $P$ is clear from the context.)
Down-sets of this form are called {\em principal}.
If $S\subseteq X$ we say $S$ has a {\em smallest element} $s_0\in S$
if $s_0\leq s$ for all $s\in S$. Note that anti-symmetry
of $\leq$ implies that a smallest element is unique (if it exists
at all!). Similarly, we define a largest element. Moreover,
we set $$S^u = \{x\in X: x\geq s \textrm{ for all } s \in S\}$$
to be the {\em set of upper bounds} of $S$. The set
of lower bounds $S^{\ell}$ is defined analogously.

We say that a subset $S\subseteq X$ of a poset $(X,\leq)$
has an {\em infimum} or {\em largest lower bound} if
\begin{enumerate}
\item $S^{\ell} \neq \emptyset$, and
\item $S^{\ell}$ has a largest element.
\end{enumerate}
Again, an infimum (if it exists) is unique by anti-symmetry of the 
ordering relation, and it is
denoted by $\textrm{inf}(S)$ or $\bigwedge_X S$. The dual
notion (everything taken ``upside down'' in the poset)
is called {\em supremum} and is denoted by $\textrm{sup}(S)$ 
or $\bigvee_X S$. The infimum of the empty set is defined
to be the largest element of $X$ if it has one, and
the supremum is the smallest element of $X$. 

A poset $(X,\leq)$
in which infima and suprema exist for all $S\subseteq X$ 
is called a {\em complete lattice}. A {\em lattice}
has suprema and infima for finite non-empty subsets.
If $(X,\leq)$ is a poset and $x,y\in X$ we use the 
following notation $$x\vee y := \bigvee_X\{x,y\},$$ and
$x\wedge y$ is defined analogously. To emphasize the
binary operations $\vee,\wedge$, a lattice $(L,\leq)$
is sometimes written as $(L,\vee,\wedge)$. A lattice
is {\em distributive} if for all $x,y,z\in L$ we have
$$x\wedge(y\vee z) = (x\vee y)\wedge(x\vee z).$$

\begin{definition}
Given a poset $(X,\leq)$, the {\em interval topology} 
$\tau_i(X)$ is
given by the subbase $${\mathcal S} = \{X\setminus 
(\downarrow x): x\in X\} \cup \{X\setminus (\uparrow x):
x\in X\}.$$
\end{definition}

Finally we give the notion of breadth of a lattice.
\begin{definition} Let $n\in\mathbb{N}$ be a positive
integer.
For a complete lattice $(L, \leq)$ we say that it has 
{\em breadth} $\leq n$ 
if for any finite set $F$ there is $A\subseteq F$ with $|A| \leq n$
such that $\inf(A) = \inf(F)$. We say $L$ has {\em finite breadth}
if there is a positive integer $n\in\mathbb{N}$ such that $L$
has breadth $\leq n$. Otherwise we say that $L$ has infinite
breadth.
\end{definition}

% . . . . . . . . . . . . . . . . .
\section{Problem 85}
Here is the statement of this problem:
\begin{quote}
Is every complete morphism (i.e., for arbitrary joins and meets)
of complete lattices continuous with respect to star-convergence?
in the interval topology?
\end{quote}
For the notion of star-convergence, we have to introduce
some further notions. We start with the answer to the second part of the question, which
is about the interval topology.

\subsection{Interval topology}
\begin{proposition}
A complete homomorphism between complete lattices is continuous in the interval topology.
\end{proposition}
\begin{proof}
Let $L$ and $M$ be complete lattices and let 
$f:L\to M$ be a complete lattice homomorphism.  Then $f$ 
is order-preserving. Let $x,y \in M$ be such that 
$x\leq y$. Either $f^{-1}([x,y])$ is empty or else take 
$a=\inf(f^{-1}([x,y]))$ and $b=\sup(f^{-1}([x,y]))$.  Then $a\leq b$.  
Then $$f(a)=f(\inf(f^{-1}([x,y]))) = \inf(f( f^{-1}([x,y]))) \geq x,$$
and similarly $$f(b)= f(\sup(f^{-1}([x,y]))) = 
\sup(f( f^{-1}([x,y]))) \leq y,$$
and $f(a)\leq f(b)$ so $x\leq f(a) \leq f(b) \leq y$ and hence 
$f(a), f(b)$ are in $[x,y]$.

Let $c \in [a,b]$.  Then $x\leq f(a) \leq f(c) \leq f(b) \leq y$, 
so $[a,b]$ is a subset of $f^{-1}([x,y])$.  
But $f^{-1}([x,y])$ is a subset of $[a,b]$. Thus $f^{-1}([x,y]) = [a,b]$.

Hence $f^{-1}$ takes subbasic closed sets to subbasic closed sets or the empty set. 
Therefore $f$ is continuous when $L$ and $M$ have the interval topology.
\end{proof}
\subsection{Star-convergence}
For this part of the question we need the notion of 
order convergence expressed with filters (Birkhoff uses nets,
and filters offer an equivalent, but more concise approach to convergence
\cite{NetFil}).

Let $(P,\leq)$ be a poset. By a {\em set filter} 
$\mathcal{F}$ on $P$ we mean 
a collection of subsets of $P$ such that:
\begin{itemize}
\item[-] $\emptyset \notin \mathcal{F}$;
\item[-] $A, B\in \mathcal{F}$ implies $A\cap B\in \mathcal{F}$;
\item[-] $U\in \mathcal{F}$, $U'\subseteq P$ and $U'\supseteq U$  
implies $U'\in \mathcal{F}$.
\end{itemize}
If $\mathcal{F}$ is a set filter, then we set ${\mathcal F}^u = 
\bigcup\{F^u: F\in \mathcal{F}\}$ and define ${\mathcal F}^\ell$ similarly. 
For $x\in P$ and ${\mathcal F}$ a set filter on $P$ we write 
$${\mathcal F}\to x \textrm{ iff } \bigwedge\mathcal{F}^u = x = 
\bigvee \mathcal{F}^\ell$$ 
and say ${\mathcal F}$ {\em order-converges} to $x$.

If $\mathcal{B}$ is a collection of subsets of $P$ such that
\begin{itemize}
\item[-] $\emptyset\notin \mathcal{B}$,
\item[-] for $A, B\in\mathcal{B}$ there is $C\in\mathcal{B}$
with $C\subseteq A\cap B$,
\end{itemize}
then we call $\mathcal{B}$ a {\em filter base}. The {\em
filter generated by} $\mathcal{B}$ is the collection
of sets that contain some member of $\mathcal{B}$.

If $\mathcal{F}\subseteq\mathcal{G}$ are filters on $P$
we say that $\mathcal{G}$ is a 
{\em super-filter} of $\mathcal{F}$.

Finally, we say that a filter $\mathcal{F}$ {\em star-converges}
to $x\in P$ if for every super-filter $\mathcal{F}'$ of
$\mathcal{F}$ there is a super-filter $\mathcal{G}$ of
$\mathcal{F}'$ such that $\mathcal{F}\to x$.

If $X, Y$ are sets and $\mathcal{F}$ is a filter on $X$ 
then it is easy to verify that
$\mathcal{B}_f := \{f(F): F\in\mathcal{F}\}$ is
a filter base in $Y$. We define $f(\mathcal{F})$ to be
the filter generated by $\mathcal{B}_f$.

The positive answer to the star-convergence part of 
question 85 follows from the following two lemmas:
\begin{lemma}
Let $L, M$ be complete lattices and let $f:L\to M$ be
a complete lattice homomorphism. Suppose that $\mathcal{F}$
is a filter on $L$ and $x\in L$ such that $\mathcal{F}
\to x$. Then $f(\mathcal{F})\to f(x)$.
\end{lemma}
\begin{proof}
We prove that $\bigwedge_M f(\mathcal{F})^u = f(x)$.

The tool we use is Fact 1.1(1) from \cite{DZ}, which states that
$$ x\in \mathcal{F}^u \Leftrightarrow \downarrow x \in \mathcal{F}.$$
So assuming $\mathcal{F}\to x$ in the lattice $L$, we get $\downarrow_L x  
\in \mathcal{F}$. Therefore
$$f(\downarrow_L x) \in \mathcal{B}_f.$$ Since
$f$ is order-preserving, we get $$\downarrow_M f(x) \supseteq
f(\downarrow_L x),$$ which implies $\downarrow_M f(x)\in f(\mathcal{F})$
because $\mathcal{B}_f$ is a filter base for $f(\mathcal{F})$.
Using the other direction of the equivalence
stated above, we get $\bigwedge(f(\mathcal{F}))^u = f(x)$.

Similarly we prove that $\bigvee (f(\mathcal{F}))^\ell = f(x)$, which
implies that $f(\mathcal{F})\to f(x)$.
\end{proof}

\begin{lemma}
If $\mathcal{G} \supseteq \mathcal{F}$ are filters
on a set $X$ and $f:X\to Y$ is any map, then
$f(\mathcal{G}) \supseteq f(\mathcal{F})$.
\end{lemma}

% . . . . . . . . . . . . . . . . .
\section{Problem 87}
Here is the statement of this problem:
\begin{quote}
Can a lattice of infinite breadth be a Hausdorff lattice
in its interval topology?
\end{quote}
We will show that $2^\omega$ is such an example. (We order 
$2 = \{0,1\}$ by $0<1$ and set $2^\omega$ to be the set of
all functions $f:\omega \to 2$, ordered pointwise.)

First, we look at the interval topology of $2^\omega$.

\begin{lemma}
Let $(P_k)_{k\in K}$ be a family of posets.
The interval topology $\tau_i = \tau_{i}(\prod_{k\in K}P_k))$ 
on $P=\prod_{k\in K} P_k$ equals the 
product topology $\tau_p$ of the topological spaces 
$(P_k, \tau_{i}(P_k))$.
\end{lemma}

\begin{proof} Take a subbasic element of $U\in\tau_i$ and show that 
it is a member of $\tau_p$. Without loss of generality we let 
$U = P\setminus (\uparrow(x_k)_{k\in K})$ where $x_k\in P_k$. 
Note that $\uparrow(x_k)_{k\in K}$ is a product of closed sets 
in the spaces $(P_k, \tau_{i}(P_k))$, therefore it is 
closed in the product topology, so $U\in \tau_p$. 

Conversely, for some $j\in K$ 
let $U = \pi_j^{-1}(U_j)$ be subbasic in $\tau_p$
where $\pi_j: P\to P_j$ is the projection map and 
$U_j = P_j\setminus \uparrow x^*$ for some $x^*\in P_j$. 
Then $$U = \bigcup \{P\setminus (\uparrow (z_k)_{k\in K}): 
(z_k)_{k\in K} \in P \text{ and } z_j =x^*\}.$$
So $U\in\tau_i$.
\end{proof}

\begin{corollary}\label{Hausdorff} 
The interval topology on $2^\omega$ is
Hausdorff.\end{corollary}
\begin{proof}The lemma shows that the interval
topology is just the product topology
of the (discrete) Hausdorff topology on $2=\{0,1\}$,
and the product topology of Hausdorff spaces
is always Hausdorff.\end{proof}

\begin{lemma}\label{infbreadth}
The complete lattice $2^\omega$ has infinite breadth.
\end{lemma}
\begin{proof} For $m\in\omega$ we let $e_m:\omega \to 2=\{0,1\}$
be the function where $e_m(m) = 0$ and $e_m(k) = 1$
for $m\neq k$.

In order to show that for any positive $n\in\mathbb{N}$
the complete lattice $2^\omega$ does not have
breadth $\leq n$, we consider the finite set $$F = \{e_0, \ldots, e_n\}.$$

So $\inf(F)\in 2^\omega$ is the function $r:\omega\to 2$
such that $r(k) = 0$ for $k \leq n$ and $r(k) = 1$ otherwise.

Note that $F$ has $n+1$ elements, and that
for no subset of $A\subseteq F$ with $A\neq F$
do we have $\inf(A) = \inf(F)$.
\end{proof}

So corollary \ref{Hausdorff} and lemma \ref{infbreadth}
answer question 87.

%................... bibliography .................
%-- end footnotesize

{\footnotesize
\begin{thebibliography}{99}
%\bibitem{Er} \Erd, Paul, {\it On the combinatorial problems which I would
%most like to see solved}, Combinatorica {\bf 1} (1981), 25--42.
\bibitem{Bir} G.~Birkhoff, {\bf Lattice Theory}, third edition, p. 253.
\bibitem{DZ} D.~van der Zypen, {\it Order convergence and compactness},  
Cah.~Topol.~Geom.~Diff.~Cat. (2004), 45(4), 297--300.
\bibitem{NetFil} \url{https://en.wikipedia.org/wiki/Net_(mathematics)#Relation_to_filters}
\end{thebibliography}
}
\end{document}